\documentclass[11pt, reqno, dvipsnames]{amsart}

\usepackage[english]{babel}
\usepackage[utf8]{inputenc}

\usepackage[pdftex]{graphicx}
\usepackage{amsmath}
\usepackage{amssymb}
\usepackage{amsfonts}
\usepackage{amsthm}
\usepackage[all,cmtip]{xy}
\usepackage{tikz}
\usepackage{tikz-cd}
\usepackage[new]{old-arrows}
\usepackage{blindtext}
\usepackage{accents}
\usepackage{scalerel}
\usepackage{stmaryrd}
\usepackage{mathrsfs}
\usepackage{dsfont}
\usepackage{relsize}
\usepackage{url}
\usepackage{contour}
\usepackage{ulem}
\usepackage{footmisc}
\usepackage{enumerate}
\usepackage[T1]{fontenc}
\usepackage{fancyhdr}

\usetikzlibrary{decorations.markings}
\tikzset{double line with arrow/.style args={#1,#2}{decorate,decoration={markings, mark=at position 0 with {\coordinate (ta-base-1) at (0,1pt);
\coordinate (ta-base-2) at (0,-1pt);},
mark=at position 1 with {\draw[#1] (ta-base-1) -- (0,1pt);
\draw[#2] (ta-base-2) -- (0,-1pt);
}}}}

\usepackage{mathtools}

\usetikzlibrary{trees}

\tikzstyle{level 1}=[level distance=2.4cm, sibling distance=6.5cm]
\tikzstyle{level 2}=[level distance=2.4cm, sibling distance=2.5cm]
\tikzstyle{level 3}=[level distance=3cm, sibling distance=0.8cm]

\usepackage[linktocpage=true]{hyperref}
\hypersetup{
    colorlinks = true,
    linkcolor = {ForestGreen},
    citecolor = {RoyalBlue},
    urlcolor = {Black}
}

\newcommand{\listref}[1]{{\color{ForestGreen}(}\ref{#1}{\color{ForestGreen})}}

\newcommand{\Zz}{\mathbb{Z}}
\newcommand{\Ff}{\mathbb{F}}

\newcommand{\Qq}{\mathbb{Q}}
\newcommand{\Hh}{\mathbb{H}}

\newcommand{\Gg}{\mathbb{G}}

\newcommand{\solidif}{\scalebox{0.5}{$\blacksquare$}}

\newcommand{\cl}[1]{\mathcal{#1}}
\newcommand{\fr}[1]{\mathfrak{#1}}

\contourlength{0.8pt}

\DeclareMathOperator*{\colim}{colim}
\DeclareMathOperator{\ett}{\acute{e}t}
\DeclareMathOperator{\pet}{pro\acute{e}t}

\DeclareMathOperator{\Hom}{Hom}

\DeclareMathOperator{\Pic}{Pic}
\DeclareMathOperator{\Spec}{Spec}

\DeclareMathOperator{\Fil}{Fil}

\DeclareMathOperator{\cond}{cond}

\DeclareMathOperator{\CondAb}{CondAb}

\DeclareMathOperator{\disc}{disc}

\DeclareMathOperator{\crys}{crys}

\DeclareMathOperator{\syn}{syn}
\DeclareMathOperator{\fib}{fib}

\DeclareMathOperator{\Spf}{Spf}

\DeclareMathOperator{\QSyn}{QSyn}
\DeclareMathOperator{\QRSPerfd}{QRSPerfd}

\DeclareMathOperator{\length}{length}
\DeclareMathOperator{\tor}{tors}

\usepackage{relsize}
\usepackage[bbgreekl]{mathbbol}
\DeclareSymbolFontAlphabet{\mathbb}{AMSb} 
\DeclareSymbolFontAlphabet{\mathbbl}{bbold}
\newcommand{\Prism}{{\mathlarger{\mathbbl{\Delta}}}}

\numberwithin{equation}{section}
\newtheorem{theorem}{Theorem}
\numberwithin{theorem}{section}

\newtheorem{lemma}[theorem]{Lemma}
\newtheorem{cor}[theorem]{Corollary}
\newtheorem{prop}[theorem]{Proposition}
\newtheorem{conj}[theorem]{Conjecture}

\theoremstyle{definition}

\newtheorem{convnot}[theorem]{Notation and conventions}
\newtheorem{notation}[theorem]{Notation}

\newtheorem{rmk}[theorem]{Remark}

\newtheorem*{Acknowledgments}{Acknowledgments}

\theoremstyle{remark}

 \AtBeginDocument{%
 \def\MR#1{}
}

\date{\today}

\makeatletter
\renewcommand{\address}[1]{\gdef\@address{#1}}
\renewcommand{\email}[1]{\gdef\@email{\url{#1}}}
\newcommand{\@endstuff}{\par\vspace{\baselineskip}\noindent\small
\begin{tabular}{@{}l}\scshape\@address\\\textit{E-mail address:} \@email\end{tabular}}
\AtEndDocument{\@endstuff}
\makeatother

\title[Torsion in $p$-adic étale cohomology: remarks and conjectures]{Torsion in $p$-adic étale cohomology: \\ remarks and conjectures}
\author{Guido Bosco}
\address{Max-Planck-Institut f\"ur Mathematik, Vivatsgasse 7, 53111 Bonn, Germany}
\email{bosco@mpim-bonn.mpg.de}

\begin{document}
 
\begin{abstract}
 
 Let $C$ be a complete algebraically closed extension of $\mathbb{Q}_p$, and let $\mathfrak{X}$ be a smooth formal scheme over $\mathcal{O}_C$. By the work of Bhatt--Morrow--Scholze, it is known that when $\mathfrak{X}$ is proper, the length of the torsion in the integral $p$-adic étale cohomology of the generic fiber $\mathfrak{X}_C$ is bounded above by the length of the torsion in the crystalline cohomology of its special fiber. In this note, we focus on the non-proper case and observe that when $\mathfrak{X}$ is affine, the torsion in the integral $p$-adic étale cohomology of $\mathfrak{X}_C$ can even be expressed as a functor of the special fiber, unlike in the proper case.  As a consequence, we show that, surprisingly, if $\mathfrak{X}$ is affine, the integral $p$-adic étale cohomology groups of $\mathfrak{X}_C$ have finite torsion subgroups. We discuss further applications and propose conjectures predicting the torsion in the integral $p$-adic étale cohomology of a broader class of rigid-analytic varieties over $C$.

\end{abstract}

\maketitle

\setcounter{tocdepth}{2}

\tableofcontents

\section{Introduction}
 
 Let $p$ be a prime number. We denote by $C$ an algebraically closed complete non-archimedean extension of $\Qq_p$ and  we write $\cl O_C$ for its ring of integers. \medskip
 
 In recent years, there has been increasing interest in the study of the integral/rational $p$-adic (pro-)étale cohomology of rigid-analytic varieties over $C$ that are not necessarily proper, such as local Shimura varieties. See e.g. the works \cite{CDN0}, \cite{CDN1}, \cite{CDNint}, \cite{Bosco}, \cite{Bosco2}. One difficulty in the latter works is that, in general, the rational $p$-adic étale cohomology groups of non-proper rigid-analytic varieties over $C$ are not finite-dimensional $\Qq_p$-vector spaces. Consequently, it becomes important to understand whether such cohomology groups have a reasonable topological structure, or, at the integral level, to understand the torsion in the integral $p$-adic étale cohomology of such spaces. \medskip 
 
 Let us begin with the building blocks of rigid-analytic varieties, namely affinoid spaces. While writing \cite{Bosco2}, and in particular thinking about \cite[Conjecture 1.15]{Bosco2},  we arrived from several directions at the following conjecture.

 \begin{conj}\label{mainconj}
  Let $X$ be a smooth affinoid rigid space over $C$. 
  \begin{enumerate}[(i)]
   \item\label{mainconj1} For all $i\in \Zz$, the torsion subgroup of the pro-étale cohomology group $H^i_{\pet}(X, \Zz_p)$ is finite.
   \item\label{mainconj2} For all $n\in \Zz_{\ge 0}$, the $p^n$-torsion subgroup of the Picard group $\Pic(X)$ is finite.
  \end{enumerate}
 \end{conj}

 We note that Conjecture \ref{mainconj}\listref{mainconj1} implies in particular that, for $X$ a smooth affinoid rigid space over $C$, the cohomology group $H^i_{\pet}(X, \Qq_p)$ has a natural structure of $\Qq_p$-Banach space (see Proposition \ref{banacchino}), therefore it is not pathological as one might think. \medskip
 
 The first goal of this note is to prove Conjecture \ref{mainconj} in the good reduction case. More precisely:
 
 \begin{theorem}[Theorem \ref{sur}, Corollary \ref{picaffinoid}]\label{mainfin}
  Conjecture \ref{mainconj} holds true in the case $X$ has an affine smooth $p$-adic formal model over $\cl O_C$.
 \end{theorem}

 We prove Theorem \ref{mainfin} by reducing it to a problem in characteristic $p$. For the first part of the statement, i.e. in order to prove Conjecture \ref{mainconj}\listref{mainconj1} in the good reduction case, we will rely on the following result expressing the $p^n$-torsion in the integral $p$-adic étale cohomology in terms of the special fiber. \medskip
 
 We denote by $k$ the residue field of $C$.

 \begin{theorem}[Theorem \ref{comparison}, Theorem \ref{comparisonStein}]\label{mainobs}
  Let $\fr X$ be a smooth $p$-adic formal scheme over $\cl O_C$ that is affine, or, more generally, can be written as a countable increasing union of affine opens.  Then, for all $i\in \Zz$ and $n\in \Zz_{\ge 0}$ we have a natural isomorphism
  $$H^{i+1}_{\pet}(\fr X_{C}, \Zz_p)[p^n]\cong H^1(\fr X_k, W\Omega_{\log}^i)[p^n].$$
  Here, we denote by $W\Omega_{\log}^i$ the logarithmic de Rham--Witt sheaf (see Notation \ref{dRW}).
 \end{theorem}

 \begin{rmk}[Comparison to the proper case]
  We note that Theorem \ref{mainobs} cannot hold for any smooth $p$-adic formal scheme $\fr X$ over $\cl O_C$. In fact, in the case $\fr X$ is proper, Bhatt--Morrow--Scholze proved that (cf. \cite[Theorem 1.1, (ii)]{BMS1}), for all $i, n\in \Zz_{\ge 0}$, we have the inequality
 \begin{equation*}
 \length_{\Zz_p}(H^i_{\pet}(\fr X_C, \Zz_p)[p^n])\le \length_{W(k)}(H^i_{\crys}(\fr X_k/W(k))[p^n]),
 \end{equation*}
  but, the torsion in $H^i_{\pet}(\fr X_C, \Zz_p)$ is not a functor of the special fiber, \cite[Remark 2.4]{BMS1}. Therefore, Theorem \ref{mainobs} is somewhat orthogonal to the cited results in the proper case.
 \end{rmk}
 
 The proof of Theorem \ref{mainobs} in the affine case relies on the work \cite{BMS2}, in particular on the comparison between $p$-adic étale and syntomic cohomology, and crucially uses a rigidity result for the syntomic cohomology proved by Antieau--Mathew--Morrow--Nikolaus, \cite{AMMN}. Using Theorem \ref{mainobs} in the affine case, we prove Theorem \ref{mainfin} by employing results on the structure of logarithmic de Rham--Witt cohomology due to Illusie--Raynaud, \cite{IR}. On the other hand, the extension of Theorem \ref{mainobs} beyond the affine case relies on the same finiteness results shown in the proof of Theorem \ref{mainfin}. \medskip

 We expect that one can prove Conjecture \ref{mainconj} in the semistable reduction case using a logarithmic variant of the results in this paper.

 \begin{rmk}[On Drinfeld's upper half-spaces] In light of the above, it is natural to conjecture, in particular, an extension of Theorem \ref{mainobs} to the semistable reduction case; see Conjecture \ref{conjst} for a precise statement. More generally, one could ask for a version of Theorem \ref{mainobs} for $\fr X$ a semistable $p$-adic formal scheme over $\cl O_C$ whose generic fiber $\fr X_C$ can be written as a countable increasing union of open affinoids. The appeal of a positive answer to the latter question lies in the fact that it could be applied, for example, to Stein rigid spaces over $C$ admitting a semistable formal model over $\cl O_C$, such as the $d$-dimensional Drinfeld's upper half-space $\Hh_C^d$ over $C$, for $d \in \Zz_{\ge 1}$. In this example, this would lead to a conceptual proof of a result by Colmez--Dospinescu--Nizioł, \cite{CDNint}, which states that 
 \begin{equation}\label{tff} 
 H_{\pet}^i(\Hh_C^d, \Zz_p) \text{ is torsion-free for all } i \ge 0. \end{equation} In fact, denoting by $\fr H$ the standard semistable formal model over $\cl O_C$ of $\Hh_C^d$, \cite[\S 6.1]{GKdri}, \cite[\S 5.1]{CDN1},\footnote{More precisely, $\fr H$ is the base change to $\cl O_C$ of the semistable formal model over $\Zz_p$ of the $d$-dimensional Drinfeld's upper half-space over $\Qq_p$ considered in \textit{loc. cit.}.} by a result of Gr\"osse-Klonne (see \cite[Corollary 6.25(2)]{CDN1}), we have 
 $$H^j(\fr H_k/k^0,  W\Omega_{  \log}^i)=0 \text{ for all } j>0 \text{ and } i\ge 0$$
 (see Notation \ref {notss} for the log structures we are considering here).
 In \cite{CDNint}, the result (\ref{tff}) is deduced from an explicit computation of $H_{\pet}^i(\Hh_C^d, \Zz_p)$, \cite[Theorem 1.1]{CDNint}; on the other hand, as explained in the introduction of \textit{op. cit.}, using  (\ref{tff}), this computation can be immediately deduced from its rational variant, previously treated in \cite[Theorem 6.55]{CDN1}.
 \end{rmk}

 We conclude recalling that, according to a folklore conjecture, any smooth affinoid rigid space over $C$ should have a semistable $p$-adic formal model over $\cl O_C$. However, it would be interesting to find a proof of Conjecture \ref{mainconj} working purely in terms of rigid-analytic varieties. We hope that, in the future, tools from integral $p$-adic Hodge theory of rigid-analytic varieties will become available to help address these questions.

 \begin{Acknowledgments}
The main ideas behind this note emerged during conversations with Matthew Morrow in the fall of 2021. However, I was only able to complete the proof of Theorem \ref{mainfin} recently, after coming across Illusie–Raynaud’s result \cite[II, Corollaire 3.11]{IR}. I warmly thank Matthew for his help and encouragement. I am grateful to Alexander Petrov for many helpful comments and for suggesting Lemma \ref{derived-p}, which led to a simplification of my previous argument. I also thank Gabriel Dospinescu, Akhil Mathew, and Peter Scholze for comments or corrections on a previous version of this note and for subsequent discussions. Additionally, I would like to thank Marco D’Addezio, Ian Gleason,  and Wiesława Nizioł for helpful conversations.
 \end{Acknowledgments}

 \begin{convnot}\label{cvv}
  Let $p$ be a prime number.
  Unless otherwise stated, we denote by $C$ an algebraically closed complete non-archimedean extension of $\Qq_p$, and we write $\cl O_C$ for its ring of integers. \medskip
  
  All formal schemes are $p$-adic and locally of finite type over the base. Given a formal scheme $\fr X$ over $\cl O_C$, we denote by $\fr X_C$ its generic fiber regarded as an adic space. In general, all rigid-analytic varieties will be regarded as adic spaces. \medskip
  
  We fix an uncountable cardinal $\kappa$ as in \cite[Lemma 4.1]{Scholze3}.  We denote by $\CondAb$ the category of $\kappa$-condensed abelian groups (and the prefix ``$\kappa$'' will often be tacit), \cite[Definition 2.1]{Scholzecond}. \medskip
  
  For  an adic space $X$  over $\Qq_p$, we denote by $X_{\pet}$ its $\kappa$-small pro-\'etale site.
  Given  a sheaf of abelian groups $\cl F$ on $X_{\pet}$, we regard, especially in \S \ref{esp}, the pro-\'etale cohomology complex $R\Gamma_{\pet}(X, \cl F):=R\Gamma(X_{\pet}, \cl F)$ in the derived category $D(\CondAb)$ (see also \cite[Definition 2.7]{Bosco}).
  We adopt similar notation and conventions for the $\kappa$-small pro-\'etale site of a (formal) scheme. \medskip
  
  For an abelian group $A$ and $m\in \Zz_{\ge 1}$, we denote by $A[m]\subseteq A$ the kernel of the multiplication by $m$ on $A$ 
 \end{convnot}

\section{Rigidity of torsion}\label{mc}
 
 Our starting point is the following result.
 
 \begin{theorem}\label{comparison}
  Let $\fr X$ be an affine smooth formal scheme over $\cl O_C$. Let $k$ denote the residue field of $C$. Then, for all $i\in \Zz$ and $n\in \Zz_{\ge 0}$ we have a natural isomorphism
  $$H^{i+1}_{\pet}(\fr X_{C}, \Zz_p)[p^n]\cong H^1(\fr X_k, W\Omega_{\log}^i)[p^n].$$
  Here, we denote by $W\Omega_{\log}^i$ the logarithmic de Rham--Witt sheaf (see Notation \ref{dRW}).
 \end{theorem}

 We will prove Theorem \ref{comparison} via comparison with syntomic cohomology, using crucially a rigidity result for the latter, due to Antieau--Mathew--Morrow--Nikolaus, \cite[Theorem 5.2]{AMMN}. 
 
 \subsubsection{\normalfont{\textbf{$p$-adic Tate twists}}}
 
 We begin by recalling some crucial results, due to Bhatt--Morrow--Scholze \cite{BMS2}, on the quasisyntomic sheaves $\Zz_p(i)^{\syn}$.
 
 \begin{notation}
 Let $\QSyn$ denote the quasisyntomic site, \cite[\S 4]{BMS2}, and let $\QRSPerfd\subset \QSyn$ denote its basis consisting of the quasiregular semiperfectoid rings, \cite[Proposition 4.21]{BMS2}. \medskip 
 
  For $i\in \Zz_{\ge 0}$,  we denote by $\Zz_p(i)^{\syn}$ the sheaf on $\QSyn$ introduced in \cite[\S 7.4]{BMS2}, defined for $R\in \QRSPerfd$ by
  $$\Zz_p(i)^{\syn}(R):=\fib(\Fil_{\cl N}^i\widehat\Prism_{R}\{i\}\xrightarrow{\varphi_i-1}\widehat\Prism_{R}\{i\})$$
  where, in terms of \cite{Prisms} and \cite{BL1}, we denote by $\widehat\Prism_{R}\{i\}$ the Breuil--Kisin twisted Nygaard completed absolute prismatic cohomology of $R$, we write $\Fil_{\cl N}^\star$ for the Nygaard filtration (\cite[Theorem 1.12(3)]{BMS2}, \cite[Theorem 13.1]{Prisms}), and we denote by $\varphi_i$ the $i$-th divided Frobenius map. \medskip
  
 For $n\in \Zz_{\ge 1}$, we define the sheaf  on $\QSyn$
 $$\Zz/p^n(i)^{\syn}:=\Zz_p(i)^{\syn}/p^n.$$
 \end{notation}

 In mixed characteristic, the sheaves defined above can be compared to $p$-adic nearby cycles, as we now recall.
 In the following, given a smooth formal scheme $\mathfrak X$ over $\cl O_C$, we regard $\Zz_p(i)^{\syn}$ on the pro-étale site $\mathfrak X_{\pet}$ using \cite[Remark 10.4]{BMS2}.
 
 \begin{theorem}[{\cite[Theorem 10.1]{BMS2}}]\label{B2} Let $\mathfrak X$ be a smooth formal scheme over $\cl O_C$ with generic fiber $X$. Denote by $\psi: X_{\pet}\to \mathfrak X_{\pet}$ the natural morphism of sites. Then, there is a natural isomorphism of sheaves of complexes on the pro-étale site $\mathfrak X_{\pet}$
 $$\Zz_p(i)^{\syn}\simeq \tau^{\le i} R\psi_* \Zz_p(i)$$
 and, for all $n\in \Zz_{\ge 1}$, there is a natural isomorphism
 $$\Zz/p^n(i)^{\syn}\simeq \tau^{\le i} R\psi_* \Zz/p^n(i).$$
 \end{theorem}
 
 Moreover, in equal characteristic $p$, the sheaves $\Zz_p(i)^{\syn}$ can be compared to the logarithmic de Rham--Witt sheaves. To state this result precisely, we introduce the following notation.
 
 \begin{notation}\label{dRW}
  Let $k$ be a perfect field of characteristic $p$, and let $X$ be a smooth scheme over $k$. 
  
  We denote by $$W\Omega^\bullet=\lim_n W_n\Omega^\bullet \;\;\; \text{ on } \;\; X_{\pet}$$ the \textit{de Rham--Witt complex} of Bloch--Deligne--Illusie of $X$, \cite[\S I.1]{Illusie}, regarded as a complex of pro-étale sheaves on $X$. It comes equipped with a Frobenius operator $F$ and a Verschiebung operator $V$ satisfying $FV=VF=p$.
  We denote by $$W\Omega_{\log}^\bullet=\lim_n W_n\Omega_{\log}^\bullet \;\;\; \text{ on } \;\; X_{\pet}$$ the \textit{logarithmic de Rham--Witt complex} of $X$, \cite[\S II.5.7]{Illusie}, fitting in the following exact sequence of complexes of pro-étale sheaves on $X$ 
  $$0\to W\Omega_{\log}^\bullet\longrightarrow W\Omega^\bullet\overset{F-1}{\longrightarrow}W\Omega^\bullet\to 0,$$
  \cite[I, Théorème 5.7.2]{Illusie}. 
 \end{notation}

 \begin{rmk}
  We recall from  \cite[Proposition 8.4]{BMS2} that we have isomorphisms $$W\Omega^i\cong R\lim_n W_n\Omega^i, \;\;\;\;\;\;\;\;\; W\Omega_{\log}^i\cong R\lim_n W_n\Omega_{\log}^i$$
  on $X_{\pet}$.
 \end{rmk}

 We will use several times the following result of Illusie.
 
 \begin{lemma}[{\cite[I, Corollaire 5.7.5]{Illusie}}]\label{logmod}
  Let $k$ be a perfect field of characteristic $p$, and let $X$ be a smooth scheme over $k$. For all $n\in \Zz_{\ge 1}$ the natural map
  $$W\Omega_{\log}^i /p^n \to W_n\Omega_{\log}^i$$
  is an isomorphism of pro-étale sheaves on $X$.
 \end{lemma}

 In the following, we regard $\Zz_p(i)^{\syn}$ on the pro-étale site $X_{\pet}$.

 \begin{theorem}[{\cite[Corollary 8.21]{BMS2}}]\label{B3}
  Let $k$ be a perfect field of characteristic $p$,  and let $X$ be a smooth scheme over $k$. Then, there is a natural isomorphism of sheaves of complexes on the pro-étale site $X_{\pet}$
  \begin{equation}\label{c1}
   \Zz_p(i)^{\syn}\simeq W\Omega_{\log}^i[-i]
  \end{equation}
  and, for all $n\in \Zz_{\ge 1}$, there is a natural isomorphism
  \begin{equation}\label{c2}
   \Zz/p^n(i)^{\syn}\simeq W_n\Omega_{\log}^i[-i].
  \end{equation}
 \end{theorem}
 \begin{proof}
  The isomorphism (\ref{c1}) is \cite[Corollary 8.21]{BMS2}. Reducing (\ref{c1}) modulo $p^n$, we obtain the isomorphism (\ref{c2}), thanks to Lemma \ref{logmod}.
 \end{proof}

 \subsubsection{\normalfont{\textbf{Rigidity of syntomic cohomology}}}
 Let $i\in \Zz_{\ge 0}$. By \cite[Theorem 5.1, (2)]{AMMN}, one can extend the definition of the \textit{integral syntomic cohomology} $$\Zz_p(i)^{\syn}(R)=R\Gamma_{\syn}(R, \Zz_p(i))$$ to all $p$-adically complete rings $R$, via left Kan extension from finitely generated $p$-adically complete polynomial $\Zz_p$-algebras. \medskip
 
 We can then state the following rigidity result.
 
 \begin{theorem}[{\cite[Theorem 5.2]{AMMN}}]\label{B4}
  Let $(R, I)$ be a henselian pair where $R$ and $R/I$ are $p$-adically complete rings. Then, the homotopy fiber
  $$\fib\left(\Zz_p(i)^{\syn}(R)\to \Zz_p(i)^{\syn}(R/I)\right)$$
  is concentrated in cohomological degrees $\le i$.
 \end{theorem}
 
 As an immediate consequence, we have the following. We will switch from algebraic to geometric notation to denote syntomic cohomology.
 
 \begin{cor}\label{B5}
  Let $\fr X=\Spf(R)$ be an affine formal scheme over $\cl O_C$. Denote by $k$ the residue field of $C$. Then, the homotopy fiber
  $$\fib\left(\Zz_p(i)^{\syn}(\fr X)\to \Zz_p(i)^{\syn}(\fr X_{k})\right)$$
  is concentrated in cohomological degrees $\le i$.
 \end{cor}
 \begin{proof}
   Denote by $\mathfrak m_C$ the maximal ideal of $\cl O_C$, so that $k=\cl O_C/\mathfrak m_C$ is the residue field of $C$. We observe that the pair $(R, \mathfrak m_C R)$ is henselian: by \cite[Tag 0DYD]{Thestack}, we can reduce to checking that $(R, pR)$ and $(R/p, \mathfrak m_C R/p)$ are henselian pairs; for the first one, by \cite[Tag 0ALJ]{Thestack}, it suffices to note that $R$ is $p$-adically complete; for the second pair, we observe that the ideal $\mathfrak m_CR/p$ of $R/p$ is locally nilpotent, and then we apply \cite[Tag 0ALI]{Thestack}. Thus, the statement follows by applying Theorem \ref{B4} to the henselian pair $(R, \mathfrak m_C R)$.
 \end{proof}

 We are ready to prove the main result of this section.
 
 \begin{proof}[Proof of Theorem \ref{comparison}] Considering the short exact sequences 
 \begin{equation*}
  0 \to H^i_{\pet}(\fr X_C, \Zz_p(i))/p^n\to H^i_{\pet}(\fr X_C, \Zz/p^n(i)) \to H^{i+1}_{\pet}(\fr X_C, \Zz_p(i))[p^n]\to 0
 \end{equation*}
 and 
 \begin{equation*}
  0 \to H^i(\Zz_p(i)^{\syn}(\fr X))/p^n\to H^i(\Zz/p^n(i)^{\syn}(\fr X)) \to H^{i+1}(\Zz_p(i)^{\syn}(\fr X))[p^n]\to 0,
 \end{equation*}
 using Theorem \ref{B2}, we deduce that we have a natural isomorphism
 \begin{equation}\label{pet-syn}
  H^{i+1}_{\pet}(\fr X_C, \Zz_p(i))[p^n]\cong H^{i+1}(\Zz_p(i)^{\syn}(\fr X))[p^n].
 \end{equation}
 On the other hand, by Corollary \ref{B5}, we have a natural isomorphism
 \begin{equation}\label{syn-syn}
  H^{i+1}(\Zz_p(i)^{\syn}(\fr X))\cong H^{i+1}(\Zz_p(i)^{\syn}(\fr X_k)).
 \end{equation}
 Therefore, combining (\ref{pet-syn}), (\ref{syn-syn}), and using Theorem \ref{B3}, we deduce that
 \begin{equation}\label{almost}
  H^{i+1}_{\pet}(\fr X_C, \Zz_p(i))[p^n]\cong H^{1}(\fr X_k, W\Omega_{\log}^i)[p^n].
 \end{equation}
 We obtain the statement observing that, as $C$ is algebraically closed, we can ignore the twist $(i)$ on the left hand side of (\ref{almost}).
 \end{proof}

 \section{Preliminaries on log de Rham--Witt cohomology}\label{pr}
 
 \begin{notation}
  In this section, we denote by $k$ an algebraically closed field of characteristic $p$.
 \end{notation}

 Our goal here is to collect a number of known results on the logarithmic de Rham--Witt cohomology of smooth proper schemes over $k$ that will be useful in the next section, when combined with Theorem \ref{comparison}.  \medskip
 
 We start with the following results of Illusie--Raynaud on the structural properties of the logarithmic de Rham--Witt cohomology.
  
 \begin{lemma}[{\cite[IV, Corollaire 3.5(a)]{IR}}]\label{short-log}
   Let $X$ be a smooth proper scheme over $k$. For all $i, j\in \Zz$, we have a short exact sequence 
    \begin{equation*}
    0\to H^j(X, W\Omega_{\log}^i)\longrightarrow H^j(X, W\Omega^i)\overset{F-1}{\longrightarrow} H^j(X, W\Omega^i)\to 0.
   \end{equation*}
 \end{lemma}

  \begin{rmk}
   Given  a smooth proper scheme $X$ over $k$, in \cite[\S IV.3]{IR} the cohomology group  $H^j(X, W\Omega_{\log}^i)$ denotes $\lim_n H^j( X, W_n\Omega_{\log}^i)$. However, it is also shown in the proof of \cite[IV, Corollaire 3.5(a)]{IR}, that the inverse system $\{H^{j-1}( X, W_n\Omega_{\log}^i)\}_n$ is Mittag-Leffler, therefore, in our notation $H^j(X, W\Omega_{\log}^i):=H^j(R\Gamma(X, W\Omega_{\log}^i))$, we have $H^j( X, W\Omega_{\log}^i)\overset{\sim}{\to} \lim_n H^j(X, W_n\Omega_{\log}^i)$.
  \end{rmk}

 \begin{lemma}[{\cite{IR}}]\label{sin}  Let $X$ be a smooth proper scheme over $k$. Let $i, j\in \Zz$.
 \begin{enumerate}[(i)]
  \item\label{sin1}  We have a natural (in $X$) short exact sequence
 \begin{equation*}\label{dev}
    0\to M^{ij}(X)\to H^j( X, W\Omega_{\log}^i)\to N^{ij}(X)\to 0 
   \end{equation*}
  where $N^{ij}(X)$ is a finitely generated $\Zz_p$-module and $M^{ij}( X)$ is killed by a power of $p$.
  \item\label{sin2} The following facts are equivalent:
  \begin{enumerate}[(a)]
   \item $M^{ij}(X)=0$,
   \item $d: H^j(X, W\Omega^{i-1})\to H^j(X, W\Omega^i)$ is the zero map,
   \item $H^j( X, W\Omega^i)^{V=0}$ is finite-dimensional over $k$, where $V$ denotes the Verschiebung operator.
 \end{enumerate}
 \item\label{sin3} The conditions in (ii) are satisfied for $j=0, 1$.
  \end{enumerate}
 \end{lemma}
 \begin{proof}
  Part \listref{sin1} and \listref{sin2} follow from  \cite[IV, Théorème 3.3(b), Corollaire 3.5(b) \& II, Corollaire 3.8]{IR}. For part \listref{sin3}, we claim that, for $j= 0, 1$, we have that $H^j( X, W\Omega^i)^{V=0}$ is finite-dimensional over $k$. In fact, for $j=0$,  by \cite[II, Corollaire 2.17]{Illusie}, we even have that $H^0(X, W\Omega^i)$ is a finite free module  over $W(k)$. For $j=1$, this is the content of (the proof of) \cite[II, Corollaire 3.11]{IR}.
 \end{proof}

 \begin{rmk}
   We note that some of the previous results were generalized to the log-smooth case by Lorenzon, \cite{Lorenzon}.

 \end{rmk}
 
 Let us now collect some consequences. For part \listref{IR1} of the result below, see also \cite[Lemma A.2]{Petrov}.

 \begin{cor}[\cite{IR}]\label{IR}
  Let $ X$ be a smooth proper scheme over $k$. Let $i\in \Zz$.
 \begin{enumerate}[(i)]
  \item\label{IR1} For all $j\in \Zz$, we have an isomorphism of $\Zz_p$-modules $$H^j(X, W\Omega_{\log}^i)\cong \Zz_p^{\oplus r}\oplus T$$ for some $r\in \Zz_{\ge 0}$, with $T$ killed by a power of $p$. 
  \item\label{IR2} For $j=0, 1$, the cohomology group $H^j(X, W\Omega_{\log}^i)$ is a finitely generated $\Zz_p$-module, and for $j=0$ it is also a free $\Zz_p$-module.
 \end{enumerate}
  \end{cor}
  \begin{proof}
   Part \listref{IR1} readily follows from Lemma \ref{sin}\listref{sin1} and the classification of finitely generated $\Zz_p$-modules.
  The first assertion of part \listref{IR2} follows combining part \listref{sin1} and \listref{sin3} of Lemma \ref{sin}. The last assertion follows recalling that $H^0(X, W\Omega_{\log}^i)=H^0( X, W\Omega^i)^{F=1}$ and, by \cite[II, Corollaire 2.17]{Illusie}, $H^0( X, W\Omega^i)$ is a finite free module over $W(k)$.
  \end{proof}
 
 \begin{rmk}
  Keeping the assumptions of Corollary \ref{IR}, by \cite[II, Proposition 5.9]{Illusie}, we also have that the cohomology group $H^2(X, W\Omega_{\log}^1)$ is a finitely generated $\Zz_p$-module. On the other hand, for $X$ a supersingular K3 surface we have that, $H^3(X, W\Omega_{\log}^1)\cong k$, \cite[\S II.7.2]{Illusie}, thus the torsion subgroup $T$ in Corollary \ref{IR}\listref{IR1} can be infinite in general.
 \end{rmk}

 \section{Finiteness of torsion and applications}\label{aff}
 
  \subsubsection{\normalfont{\textbf{Finiteness}}}

 Our first goal here is to prove the following surprising finiteness result, Theorem \ref{sur}, which thanks to Theorem \ref{comparison} can be reduced to a problem in characteristic $p$.

 \begin{theorem}\label{sur}
  Let $\fr X$ be an affine smooth formal scheme over $\cl O_C$. Then, for all $i\in \Zz$, the cohomology group $H^i_{\pet}(\fr X_C, \Zz_p)$ has finite torsion subgroup.
 \end{theorem}
 \begin{proof}
  The statement follows combining Theorem \ref{comparison} and Proposition \ref{MM}\listref{MM1} below.
 \end{proof}

 The following is the crucial technical result that we use for the proof of Theorem \ref{sur}. We will prove it by reduction to the proper case, via refined alterations and purity of the logarithmic de Rham–Witt sheaves.
 
 \begin{prop}\label{MM}
 Let $k$ be an algebraically closed field of characteristic $p$. Let $Y$ be a smooth variety over $k$.\footnote{Here, a \textit{variety over $k$} is a separated, integral, scheme of finite type over $\Spec k$.} Let $i\in \Zz$.
 \begin{enumerate}[(i)]
    \item\label{MM2} For all $n\in \Zz_{\ge 1}$, the cohomology group $H^0(Y, W_n\Omega_{\log}^i)$ is a finite abelian group.
  \item\label{MM1} The cohomology group $H^1( Y, W\Omega_{\log}^i)$ has finite torsion subgroup.
 \end{enumerate}
 \end{prop}
 
 \begin{proof}
  Part \listref{MM2} in the case $Y$ is smooth proper over $k$ follows from Corollary \ref{IR}\listref{IR2}, using the short exact sequence
  \begin{equation}\label{nat}
   0\to H^0( Y, W\Omega_{\log}^i)/p^n \to H^0( Y, W_n\Omega_{\log}^i) \to H^1( Y, W\Omega_{\log}^i)[p^n] \to 0,
  \end{equation}
  which follows from Lemma \ref{logmod}.
  
  Next, we want to show that we can reduce to the proper case. Let $Y$ be a smooth variety over $k$. Using a refinement of de Jong's alterations, \cite[Theorem 1.2]{BhattSnow}, there exist an étale morphism $ X\to  Y$ over $k$, and a dense open immersion $ X\hookrightarrow \overline{ X}$ into a smooth projective variety $\overline{ X}$ over $k$ whose complement is a strict normal crossing divisor $D \subset \overline{X}$. Since $W_n\Omega_{\log}^i$ satisfies étale descent, we have that $H^0( Y, W_n\Omega_{\log}^i)$ injects into $H^0(X, W_n\Omega_{\log}^i)$, hence it suffices to prove the statement with $X$ in place of $Y$. For this, we consider the localization exact sequence
  $$H^0_{D}(\overline{ X}, W_n\Omega_{\log}^i)\to  H^0(\overline{ X}, W_n\Omega_{\log}^i)\to H^0( X, W_n\Omega_{\log}^i)\to  H^1_{D}(\overline{ X}, W_n\Omega_{\log}^i)$$
  where $H^{\bullet}_D$ denotes the cohomology with support in the closed subscheme $D \subset  \overline{X}$.
 Then, using the purity of the logarithmic de Rham--Witt sheaves, \cite[II, Théorème 3.5.8, (3.5.19)]{Gros}, \cite[Theorem 3.1]{Shiho}, we deduce the exact sequence
  \begin{equation}\label{anvedi}
   0\to H^0(\overline{ X}, W_n\Omega_{\log}^i)\to H^0(X, W_n\Omega_{\log}^i)\to \bigoplus_{r=1}^m H^0(D_r,  W_n\Omega_{\log}^{i-1})
  \end{equation}
  where we wrote $D=\bigcup_{r=1}^m D_r$ as union of its irreducible components $D_r$.  As $\overline{X}$ and $D_r$, for $r=1, \ldots, m$, are smooth proper varieties over $k$, part \listref{MM2} follows from the previous case and the exact sequence (\ref{anvedi}).
  
  For part \listref{MM1}, let $M:=H^1( Y, W\Omega_{\log}^i)$. By part \listref{MM2} and the exact sequence (\ref{nat}) for $n=1$, we have that $M[p]$ is a finite abelian group. Observing that $M$ is derived $p$-adically complete,\footnote{In fact, $R\Gamma( Y, W\Omega_{\log}^i)=R\lim_n R\Gamma( Y, W_n\Omega_{\log}^i)$ is derived $p$-adically complete, as it is limit of $p^n$-torsion complexes. Hence, we can use \cite[Proposition 3.4.4]{BS}.} we deduce the statement from Lemma \ref{derived-p} below.
 \end{proof}

  We used the following simple lemma.

 \begin{lemma}\label{derived-p}
  Let $M$ be a derived $p$-adically complete $\Zz_p$-module. Suppose that the $p$-torsion submodule $M[p]$ is finite, then the torsion submodule $M_{\tor}=\colim_n M[p^n]$ is finite.
 \end{lemma}
 \begin{proof}
 We denote by $T_p(M):=\lim_n M[p^n]$ the $p$-adic Tate module of $M$, where the transition maps $M[p^n]\to M[p^{n-1}]$ are given by multiplication by $p$. Since $M$ is derived $p$-adically complete (and concentrated in degree 0), by \cite[Tag 0BKG]{Thestack} we have that $$T_p(M)=H^{-1}(M^{\wedge})=H^{-1}(M)=0$$
 where $(-)^{\wedge}$ denotes the derived $p$-adic completion. As $M[p]$ is finite, we deduce that there exists a sufficiently big integer $m$ such that the map $M[p^{m+1}]\to M[p]$, given by multiplication by $p^m$, is the zero map. In other words, $M_{\tor}=M[p^m]$. Now, using the exact sequence $$0\to M[p^{m-1}]\to M[p^{m}]\to M[p]$$
 where the right map is given by multiplication by $p^{m-1}$, arguing by induction we deduce that $M[p^m]$ is finite. Putting everything together, we conclude that $M_{\tor}$ is finite, as desired.
 \end{proof}
 
 \begin{rmk}
  After a first draft of this note was distributed, Alexander Petrov informed us that Proposition \ref{MM}\listref{MM2} for $n=1$ can also be deduced from \cite[Corollaire 2.5.6 (Page 534), Théorème 2.4.2 (Page 528)]{Illusie}. We note also that Proposition \ref{MM}\listref{MM2} for a general $n\in \Zz_{\ge 1}$ can be deduced from the case $n=1$. For this, we can argue by induction on $n$, using the exact sequence of pro-\'etale sheaves on $Y$, which follows from Lemma \ref{logmod}:
  $$0\to W_{n-1}\Omega_{\log}^i\to W_{n}\Omega_{\log}^i\to \Omega_{\log}^i\to 0$$
  where the right map is given by reduction modulo $p$.
 \end{rmk}
 
  We observe the following easy consequence of Theorem \ref{sur}.
 
 \begin{cor}
  Let $\fr X$ be an affine smooth formal scheme over $\cl O_C$. Then, for all $i\in \Zz$, we have a natural isomorphism
   $$H^i_{\pet}(\fr X_C, \Zz_p)\overset{\sim}{\to} \lim_n H^i_{\pet}(\fr X_C, \Zz/p^n).$$ 
 \end{cor}
 \begin{proof}
 Considering the short exact sequence
 \begin{equation*}
  0 \to H^{i-1}_{\pet}(\fr X_C, \Zz_p)/p^n\to H^{i-1}_{\pet}(\fr X_C, \Zz/p^n) \to H^{i}_{\pet}(\fr X_C, \Zz_p)[p^n]\to 0,
 \end{equation*}
 by Theorem \ref{sur} and the Mittag-Leffler criterion, \cite[Proposition 13.2.2]{Groth}, we have that $$R^1\lim_n H^{i-1}_{\pet}(\fr X_C, \Zz/p^n)=0.$$ This implies the statement.
 \end{proof}

 \subsubsection{\normalfont{\textbf{$p$-torsion in the Picard group}}}
 
 Let us now collect another interesting consequence of Proposition \ref{MM}\listref{MM2}. A proof of the following result was sketched in \cite[Lemma 2.8]{PicTors}, assuming resolution of singularities in characteristic $p$.
 
 \begin{prop}\label{pictor}
  Let $k$ be an algebraically closed field of characteristic $p$, and let $Y$ be a smooth variety over $k$. Then, for all $n\in \Zz_{\ge 0}$,
  \begin{equation*}\label{pic}
   \Pic(Y)[p^n] \text{ is finite. }
  \end{equation*}
 \end{prop} 
  \begin{proof}
  By \cite[Proposition 7.17]{BMS2}, we have a natural isomorphism
  $$\Zz_p(1)^{\syn}(Y)\simeq R\Gamma_{\ett}(Y, \Gg_m)^{\wedge}[-1]$$
  where $(-)^{\wedge}$ denotes the derived $p$-adic completion. Combining this with Theorem \ref{dRW}, we deduce that we have a short exact sequence
  $$0\to H^0(Y, \Gg_m)/p^n \to H^0(Y, W_n\Omega^1)\to \Pic(Y)[p^n]\to 0.$$
  In particular,  $H^0(Y, W_n\Omega^1)$ surjects onto $\Pic(Y)[p^n]$, and then the statement follows from Proposition \ref{MM}\listref{MM2}.
  \end{proof}

  \begin{cor}\label{picaffinoid}
    Let $\fr X$ be an affine smooth formal scheme over $\cl O_C$. Then, for all $n\in \Zz_{\ge 0}$,
  \begin{equation*}
   \Pic(\fr X_C)[p^n] \text{ is finite. }
  \end{equation*}
  \end{cor}
  \begin{proof}
   Let $k$ denote the residue field of $C$. By \cite[Lemma 6.2.4]{Lutk}, \cite[Lemma 3.6]{Heuer-good},  we have natural isomorphisms of Picard groups
   $$\Pic(\fr X_C)\cong \Pic(\fr X)\cong \Pic(\fr X_k).$$
   Then, the statement follows from Proposition \ref{pictor}.
  \end{proof}

  \subsubsection{\normalfont{\textbf{Rational $p$-adic étale cohomology of affinoids}}}\label{esp}

 As a corollary of Theorem \ref{sur}, we will deduce that, for $X$ an affinoid rigid space over $C$ having an affine smooth formal model over $\cl O_C$, the rational $p$-adic pro-étale cohomology groups of $X$ are $\Qq_p$-Banach spaces. \medskip

 The following result relies crucially on the fact that connected affinoid rigid spaces over a $p$-adic field are $K(\pi, 1)$ for $p$-torsion coefficients, as proved by Scholze.
 
 \begin{prop}\label{banacchino}
   Let $X$ be an affinoid rigid space over $C$. Let $i\in \Zz$. If $H_{\pet}^i(X, \Zz_p)$ has bounded $p^{\infty}$-torsion, then the condensed $\Qq_p$-vector space $H_{\pet}^i(X, \Qq_p)$ is a Banach space.
 \end{prop}
 
  For the proof, we will use the following general lemma. We refer the reader to \cite[\S 3.4]{BS} for the definition of derived completion in a general replete topos, as well as its basic properties, which will be used freely in the following.
  
 \begin{lemma}\label{es}
  Let $M$ be a condensed abelian group, and denote by $M(*)$ its underlying abelian group. Suppose that the following conditions are satisfied:
  \begin{enumerate}[(i)]
   \item\label{es1} the natural maps of condensed abelian groups $M\rightarrow M^{\wedge}_p \leftarrow (M(*))^{\wedge}_p$ are isomorphisms; here,  we denote $(-)^{\wedge}_p:D(\CondAb)\to D(\CondAb)$ the derived $p$-adic completion, and regard $M(*)$ as a discrete condensed abelian group;
   \item\label{es2} $M(*)$ has bounded $p^{\infty}$-torsion.
  \end{enumerate}
   Then, the condensed $\Qq_p$-vector space $M[1/p]$ is a Banach space. 
 \end{lemma}
 \begin{proof}
 By condition \listref{es1} there exists an exact sequence of condensed abelian groups as follows:
 $$\left(\bigoplus\nolimits_I \Zz\right)^{\wedge}_p\overset{f}{\to} \left(\bigoplus\nolimits_J \Zz\right)^{\wedge}_p\to M\to 0$$
 for some sets $I, J$. We write $N$ for the target of the map $f$ and $L$ for the image of $f$, so that we have a short exact sequence of condensed abelian groups 
 \begin{equation}\label{tt}
  0\to L\to N\to M\to 0.
 \end{equation}
 Using, in addition, condition \listref{es2} we deduce that the exact sequence (\ref{tt}) is given by taking the limit over $n\in \Zz_{\ge 0}$ of the exact sequence of inverse systems
 $$\{M(*)[p^n]\}\to \{L(*)/p^n\}\to \{N(*)/p^n\}\to \{M(*)/p^n\}\to 0.$$
 We deduce in particular that, inverting $p$ in (\ref{tt}), $L[1/p]$ is a closed subspace of the $\Qq_p$-Banach space $N[1/p]$, and then $M[1/p]$ is a $\Qq_p$-Banach space too, as desired.
 \end{proof}

 \begin{proof}[Proof of Proposition \ref{banacchino}]
  We may assume $X$ connected. Fix $x\in X(C)$ a base point, and let $G:=\pi_1(X, x)$ denote the profinite étale fundamental group. We claim that, for any $n\in \Zz_{\ge 1}$, there is a natural isomorphism in $D(\CondAb)$
  $$R\Gamma_{\cond}(G, \Zz/p^n)\overset{\sim}{\longrightarrow}R\Gamma_{\pet}(X, \Zz/p^n)$$
  where the left hand side denotes the condensed group cohomology (see e.g. \cite[Definition B.1]{Bosco}).
  For this, by the proof of \cite[Theorem 4.9]{Scholze},  for any $\kappa$-small extremally disconnected set $S$ (where $\kappa$ is the cardinal fixed in \ref{cvv}), the latter holds true on $S$-valued points. In particular, we have that $R\Gamma_{\pet}(X, \Zz_p)\simeq  R\Gamma_{\cond}(G, \Zz_p)$.
 Since $\Zz_p$ is a solid abelian group and $G$ is profinite, the condensed group cohomology complex $R\Gamma_{\cond}(G, \Zz_p)$ is quasi-isomorphic to the complex of solid abelian groups
  $$\Zz_p\longrightarrow \underline{\Hom}(\Zz[G]^{\solidif}, \Zz_p)\longrightarrow \underline{\Hom}(\Zz[G\times G]^{\solidif}, \Zz_p)\longrightarrow \cdots$$
  sitting in non-negative cohomological degrees (see e.g. \cite[Proposition B.2, (i)]{Bosco}).  By \cite[Corollary 6.1, (iv)]{Scholzecond} and \cite[Corollary 5.5]{Scholzecond}, for any $j\in \Zz_{\ge 0}$, using that $G^j$ is profinite, there exist a set $J$ (depending on $j$), and an isomorphism $\Zz[G^j]^{\solidif}\cong \prod_J \Zz$. Moreover, we recall that $R\underline{\Hom}(\prod_J \Zz, \Zz)=\bigoplus_J \Zz$ concentrated in degree 0 (see the proof of \cite[Proposition 5.7]{Scholzecond}); in particular, we have $$R\underline{\Hom}\left(\prod\nolimits_J \Zz, \Zz_p\right)=\left(\bigoplus\nolimits_J \Zz\right)^\wedge_p$$ concentrated in degree 0. Putting everything together, we deduce that $R\Gamma_{\pet}(X, \Zz_p)$ is quasi-isomorphic to a complex of condensed abelian groups whose terms are of the form $(\bigoplus_I \Zz)^\wedge_p$ for some set $I$. Therefore, $M := H_{\pet}^i(X, \Zz_p)$ satisfies condition \listref{es1} of Lemma \ref{es}: in fact, the objects $(\bigoplus_I \Zz)^\wedge_p$ satisfy this condition, and the condensed abelian groups meeting this condition are stable under kernels and cokernels.\footnote{In fact, the full subcategory of $\CondAb$ spanned by the objects satisfying condition \listref{es1} of Lemma \ref{es} can be identified with the essential image of the fully faithful functor from derived $p$-adically complete abelian groups to condensed abelian groups, sending $A$ to $(A^{\disc})^{\wedge}_p$, where $A^{\disc}$ denotes the condensed abelian group associated with $A$ endowed with the discrete topology. Hence, we can use that derived $p$-adically complete abelian groups are stable under kernels and cokernels, \cite[Lemma 3.4.14]{BS}.} Since, by assumption, $M$ also satisfies condition \listref{es2} of Lemma \ref{es}, the statement follows by applying the latter lemma and observing that $H_{\pet}^i(X, \Qq_p) = H_{\pet}^i(X, \Zz_p)[1/p]$, as $X$ is quasi-compact and quasi-separated.
 \end{proof}

 Combining Theorem \ref{sur} and Proposition \ref{banacchino}, we obtain the following result.
 
 \begin{theorem}\label{banacchino2}
  Let $\fr X$ be an affine smooth formal scheme over $\cl O_C$. Then, for all $i\in \Zz$, the condensed $\Qq_p$-vector space $H_{\pet}^i(\fr X_C, \Qq_p)$ is a Banach space.
 \end{theorem}

 \subsubsection{\normalfont{\textbf{Beyond the affine case}}}
 
 Our next goal is to extend Theorem \ref{comparison} to a wider class of formal schemes. 
 
 \begin{notation}
  In this section, we denote by $k$ the residue field of $C$. 
 \end{notation}

 In the proof of the following theorem, it will be crucial once again the finiteness result stated in Proposition \ref{MM}\listref{MM2}.
 
 \begin{theorem}\label{comparisonStein}
  Let $\fr X$ be a smooth formal scheme over $\cl O_C$ that admits an open affine covering $\{\fr U_r\}_{r\in \Zz_{\ge 0}}$ with $\fr U_r\subseteq \fr U_{r+1}$. Then, for all $i\in \Zz$ and $n\in \Zz_{\ge 0}$ we have a natural isomorphism
  $$H^{i+1}_{\pet}(\fr X_{C}, \Zz_p)[p^n]\cong H^1(\fr X_k, W\Omega_{\log}^i)[p^n].$$
 \end{theorem}
 \begin{proof}
  It suffices to show that Corollary \ref{B5} extends to $\fr X$ a smooth formal scheme over $\cl O_C$ as in the statement, and then we can use the same argument of the proof of Theorem \ref{comparison}. Namely, we want to show that the homotopy fiber
  $$\fib\left(\Zz_p(i)^{\syn}(\fr X)\to \Zz_p(i)^{\syn}(\fr X_{k})\right)$$
  is concentrated in cohomological degrees $\le i$. First, we note that, by the derived Nakayama lemma, the latter assertion is equivalent to its version modulo $p$, i.e. replacing $\Zz_p(i)^{\syn}$ by $\Ff_p(i)^{\syn}$. In other words, we want to prove that, for all $j\ge i+1$, the natural map $$H^j(\Ff_p(i)^{\syn}(\fr X))\to H^j(\Ff_p(i)^{\syn}(\fr X_{k}))$$
  is an isomorphism. Consider the short exact sequence
  $$0\to R^1\lim_r H^{j-1}(\Ff_p(i)^{\syn}(\fr U_r)) \to H^j(\Ff_p(i)^{\syn}(\fr X))\to \lim_r H^j(\Ff_p(i)^{\syn}(\fr U_r))\to 0$$
  and the analogous short exact sequence  with $\fr X$ replaced by $\fr X_k$ and $\fr U_r$ replaced by its special fiber $U_r:=(\fr U_{r})_k$. By Corollary \ref{B5} (applied to each affine $\fr U_r$), we deduce that it is sufficient to show that 
  \begin{equation}\label{D1}
   R^1\lim_r H^{i}(\Ff_p(i)^{\syn}(\fr U_r))=0
  \end{equation}
  and
  \begin{equation}\label{D2}
   R^1\lim_r H^{i}(\Ff_p(i)^{\syn}(U_r))=0.
  \end{equation}
  First of all, we note that (\ref{D1}) holds true if and only if (\ref{D2}) holds true. In fact, (\ref{D1}) is equivalent to $H^{i+1}(\Ff_p(i)^{\syn}(\fr U_{r}))\overset{\sim}{\to}\lim_r H^{i+1}(\Ff_p(i)^{\syn}(\fr U_{r}))$, which, in turn, again by Corollary \ref{B5}, is equivalent to (\ref{D2}). Next, we observe that, by Theorem \ref{B3}, the vanishing (\ref{D2}) is equivalent to
  \begin{equation}\label{R3}
   R^1\lim_r H^0(U_r, \Omega_{  \log}^i)=0.
  \end{equation}
  Therefore, we are reduced to check the vanishing (\ref{R3}): this follows from Proposition \ref{MM}\listref{MM2} and the Mittag-Leffler criterion, \cite[Proposition 13.2.2]{Groth}.
  \end{proof}

  At this point, it is natural to conjecture an extension of Theorem \ref{comparisonStein} to the semistable reduction case. In order to state the conjecture precisely, we need to introduce some notation.
 
  \begin{notation}\label{notss}
   We say that a formal scheme over $\cl O_C$ is \textit{semistable} if it has étale locally semistable coordinates in the sense of \cite[\S 1.5]{CK}. Given a semistable formal scheme $\fr X$ over $\cl O_C$  we endow $\fr X$ with the canonical log structure, i.e. the one given by the sheafification of the subpresheaf $\cl O_{\fr X, \ett}\cap (\cl O_{\fr X, \ett}[1/p])^\times \hookrightarrow \cl O_{\fr X, \ett}$, and the special fiber $\fr X_k$ with the pullback log structure. 
   
   We write $k^0$ for $k$ with the log structure associated with $\Zz_{\ge 0}\to  k,\, 1\,\mapsto 0,\, 0\, \mapsto 1$.
 \end{notation}

 \begin{conj}\label{conjst}
   Let $\fr X$ be a semistable formal scheme over $\cl O_C$ that admits an open affine covering $\{\fr U_r\}_{r\in \Zz_{\ge 0}}$ with $\fr U_r\subseteq \fr U_{r+1}$. Then, for all $i\in \Zz$ and $n\in \Zz_{\ge 0}$ we have a natural isomorphism
  $$H^{i+1}_{\pet}(\fr X_{C}, \Zz_p)[p^n]\cong H^1(\fr X_k/k^0, W\Omega_{ \log}^i)[p^n].$$
  Here, we use the logarithmic de Rham-Witt cohomology for log schemes as defined in \cite{Lorenzon}. 
 \end{conj}


\normalem

\bibliography{biblio}{}

\providecommand{\MR}[1]{}
\providecommand{\bysame}{\leavevmode\hbox to3em{\hrulefill}\thinspace}
\providecommand{\MR}{\relax\ifhmode\unskip\space\fi MR }
\providecommand{\MRhref}[2]{%
  \href{http://www.ams.org/mathscinet-getitem?mr=#1}{#2}
}
\providecommand{\href}[2]{#2}
\begin{thebibliography}{AMMN22}

\bibitem[AMMN22]{AMMN}
Benjamin Antieau, Akhil Mathew, Matthew Morrow, and Thomas Nikolaus, \emph{On
  the {B}eilinson fiber square}, Duke Math. J. \textbf{171} (2022), no.~18,
  3707--3806. \MR{4516307}

\bibitem[BL22]{BL1}
Bhargav Bhatt and Jacob Lurie, \emph{Absolute prismatic cohomology},
  \url{https://arxiv.org/abs/2201.06120}, 2022, Preprint.

\bibitem[BMS18]{BMS1}
Bhargav Bhatt, Matthew Morrow, and Peter Scholze, \emph{Integral {$p$}-adic
  {H}odge theory}, Publ. Math. Inst. Hautes \'{E}tudes Sci. \textbf{128}
  (2018), 219--397. \MR{3905467}

\bibitem[BMS19]{BMS2}
\bysame, \emph{Topological {H}ochschild homology and integral {$p$}-adic
  {H}odge theory}, Publ. Math. Inst. Hautes \'{E}tudes Sci. \textbf{129}
  (2019), 199--310. \MR{3949030}

\bibitem[Bos21]{Bosco}
Guido Bosco, \emph{On the $p$-adic pro-étale cohomology of {D}rinfeld
  symmetric spaces}, \url{https://arxiv.org/abs/2110.10683}, 2021, Preprint.

\bibitem[Bos23]{Bosco2}
\bysame, \emph{Rational $p$-adic {H}odge theory for rigid-analytic varieties},
  \url{https://arxiv.org/abs/2306.06100}, 2023, Preprint.

\bibitem[BS]{BhattSnow}
Bhargav Bhatt and Andrew Snowden, \emph{Refined alterations},
  \url{http://www-personal.umich.edu/~bhattb/math/alterationsepsilon.pdf}.

\bibitem[BS15]{BS}
Bhargav Bhatt and Peter Scholze, \emph{The pro-\'{e}tale topology for schemes},
  Ast\'{e}risque (2015), no.~369, 99--201. \MR{3379634}

\bibitem[BS22]{Prisms}
\bysame, \emph{Prisms and prismatic cohomology}, Ann. of Math. (2) \textbf{196}
  (2022), no.~3, 1135--1275. \MR{4502597}

\bibitem[CDN20a]{CDN0}
Pierre Colmez, Gabriel Dospinescu, and Wies{\l}awa Nizio{\l}, \emph{Cohomologie
  {$p$}-adique de la tour de {D}rinfeld{$:$} le cas de la dimension {$1$}}, J.
  Amer. Math. Soc. \textbf{33} (2020), no.~2, 311--362. \MR{4073863}

\bibitem[CDN20b]{CDN1}
\bysame, \emph{Cohomology of {$p$}-adic {S}tein spaces}, Invent. Math.
  \textbf{219} (2020), no.~3, 873--985. \MR{4055181}

\bibitem[CDN21]{CDNint}
\bysame, \emph{Integral {$p$}-adic \'{e}tale cohomology of {D}rinfeld symmetric
  spaces}, Duke Math. J. \textbf{170} (2021), no.~3, 575--613. \MR{4255044}

\bibitem[CK19]{CK}
Kestutis Cesnavicius and Teruhisa Koshikawa, \emph{The {$A_{\inf}$}-cohomology
  in the semistable case}, Compos. Math. \textbf{155} (2019), no.~11,
  2039--2128. \MR{4010431}

\bibitem[CS19]{Scholzecond}
Dustin Clausen and Peter Scholze, \emph{Lectures on {C}ondensed {M}athematics},
  \url{https://www.math.uni-bonn.de/people/scholze/Condensed.pdf}, 2019.

\bibitem[GJRW96]{PicTors}
Robert Guralnick, David~B. Jaffe, Wayne Raskind, and Roger Wiegand, \emph{On
  the {P}icard group: torsion and the kernel induced by a faithfully flat map},
  J. Algebra \textbf{183} (1996), no.~2, 420--455. \MR{1399035}

\bibitem[GK05a]{GKdri}
Elmar Grosse-Kl\"{o}nne, \emph{Frobenius and monodromy operators in rigid
  analysis, and {D}rinfel'd's symmetric space}, J. Algebraic Geom. \textbf{14}
  (2005), no.~3, 391--437. \MR{2129006}

\bibitem[GK05b]{GKint}
\bysame, \emph{Integral structures in the {$p$}-adic holomorphic discrete
  series}, Represent. Theory \textbf{9} (2005), 354--384. \MR{2133764}

\bibitem[GK07]{GKshv}
\bysame, \emph{Sheaves of bounded {$p$}-adic logarithmic differential forms},
  Ann. Sci. \'{E}cole Norm. Sup. (4) \textbf{40} (2007), no.~3, 351--386.
  \MR{2493385}

\bibitem[GK14]{GKsp}
\bysame, \emph{On special representations of {$p$}-adic reductive groups}, Duke
  Math. J. \textbf{163} (2014), no.~12, 2179--2216. \MR{3263032}

\bibitem[Gro61]{Groth}
A.~Grothendieck, \emph{\'{E}l\'{e}ments de g\'{e}om\'{e}trie alg\'{e}brique.
  {III}. \'{E}tude cohomologique des faisceaux coh\'{e}rents. {I}}, Inst.
  Hautes \'{E}tudes Sci. Publ. Math. (1961), no.~11, 167. \MR{217085}

\bibitem[Gro85]{Gros}
Michel Gros, \emph{Classes de {C}hern et classes de cycles en cohomologie de
  {H}odge-{W}itt logarithmique}, M\'{e}m. Soc. Math. France (N.S.) (1985),
  no.~21, 87. \MR{844488}

\bibitem[Heu21]{Heuer-good}
Ben Heuer, \emph{Line bundles on perfectoid covers: case of good reduction},
  \url{https://arxiv.org/abs/2105.05230}, 2021, Preprint.

\bibitem[Ill79]{Illusie}
Luc Illusie, \emph{Complexe de de {R}ham-{W}itt et cohomologie cristalline},
  Ann. Sci. \'{E}cole Norm. Sup. (4) \textbf{12} (1979), no.~4, 501--661.
  \MR{565469}

\bibitem[IR83]{IR}
Luc Illusie and Michel Raynaud, \emph{Les suites spectrales associ\'{e}es au
  complexe de de {R}ham-{W}itt}, Inst. Hautes \'{E}tudes Sci. Publ. Math.
  (1983), no.~57, 73--212. \MR{699058}

\bibitem[Lor02]{Lorenzon}
Pierre Lorenzon, \emph{Logarithmic {H}odge-{W}itt forms and {H}yodo-{K}ato
  cohomology}, J. Algebra \textbf{249} (2002), no.~2, 247--265. \MR{1901158}

\bibitem[L{\"{u}}t16]{Lutk}
Werner L{\"{u}}tkebohmert, \emph{Rigid geometry of curves and their
  {J}acobians}, Ergebnisse der Mathematik und ihrer Grenzgebiete. 3. Folge. A
  Series of Modern Surveys in Mathematics [Results in Mathematics and Related
  Areas. 3rd Series. A Series of Modern Surveys in Mathematics], vol.~61,
  Springer, Cham, 2016. \MR{3467043}

\bibitem[Pet24]{Petrov}
Alexander Petrov, \emph{Appendix to "{B}oundedness of the p-primary torsion of
  the {B}rauer group of products of varieties", by {A}lexei {N}.
  {S}korobogatov}, \url{https://arxiv.org/pdf/2404.19150}, 2024, Preprint.

\bibitem[Sch13]{Scholze}
Peter Scholze, \emph{{$p$}-adic {H}odge theory for rigid-analytic varieties},
  Forum Math. Pi \textbf{1} (2013), e1, 77. \MR{3090230}

\bibitem[Sch21]{Scholze3}
\bysame, \emph{\'{E}tale cohomology of diamonds},
  \url{https://arxiv.org/abs/1709.07343}, 2021, Preprint.

\bibitem[Shi07]{Shiho}
Atsushi Shiho, \emph{On logarithmic {H}odge-{W}itt cohomology of regular
  schemes}, J. Math. Sci. Univ. Tokyo \textbf{14} (2007), no.~4, 567--635.
  \MR{2396000}

\bibitem[{Sta}24]{Thestack}
The {Stacks Project Authors}, \emph{Stacks {P}roject},
  \url{https://stacks.math.columbia.edu}, 2024.

\end{thebibliography}
\bibliographystyle{amsalpha}

\end{document}